\documentclass[reqno,12pt]{amsart}
\usepackage{amsmath,amsthm,amssymb}
\usepackage{fullpage}
\usepackage{xcolor}
\usepackage{stmaryrd}
\begingroup
\catcode`[=\active \catcode`]=\active
\gdef[{\llbracket} \gdef]{\rrbracket}
\def\getdelim#1#2#3#4\relax{"#4}
\global\delcode`[=\expandafter\getdelim\llbracket\relax
\global\delcode`]=\expandafter\getdelim\rrbracket\relax
\endgroup
\newtheorem{thm}{Theorem}[section]
\newtheorem{lem}[thm]{Lemma}

\theoremstyle{definition}
\newtheorem{defn}[thm]{Definition}

\newtheorem{ex}[thm]{Example}

\theoremstyle{remark} \numberwithin{equation}{section}

\begin{document}
\title[]{Existence of Symmetric Positive Solutions for a Caputo Fractional Singular Boundary Value Problem}
\date{}
\author{Naseer Ahmad Asif}
\address{Department of Mathematics, School of Science, University of Management and Technology, C-II Johar Town, 54770 Lahore, Pakistan}%
\email{naseerasif@yahoo.com}%

\begin{abstract}
{In this article, we establish the symmetric positive existence for the following Caputo fractional boundary value problem
\begin{align*}
{}^{C}D_{0}^{\,\mu}x(t)+f(t,x(t))&=0,\hspace{0.4cm}t\in(-1,\,1),\hspace{0.4cm}1<\mu\leq2,\\
x(\pm1)=x'(0^{\pm})&=0,
\end{align*}
where ${}^{C}D_{0}^{\,\mu}x(t)={}^{C}D_{0^{+}}^{\,\mu}x(t)$ for $t\geq0$, ${}^{C}D_{0}^{\,\mu}x(t)={}^{C}D_{0^{-}}^{\,\mu}x(t)$ for $t\leq0$. Moreover, $f:(-1,\,1)\times(0,\infty)\rightarrow\mathbb{R}$ is continuous and singular at $t=-1$, $t=1$ and $x=0$. Here, ${}^{C}D_{0^{+}}^{\,\mu}$ and ${}^{C}D_{0^{-}}^{\,\mu}$, respectively, are Caputo fractional left and right derivatives of order $\mu$.}
\end{abstract}

\maketitle

\section{introduction}

In this article, we are concerned with the existence of symmetric positive solutions for the following Caputo fractional singular boundary value problem (SBVP)
\begin{equation}\label{cbvp}\begin{split}
{}^{C}D_{0}^{\,\mu}x(t)+f(t,x(t))&=0,\hspace{0.4cm}t\in(-1,\,1),\hspace{0.4cm}1<\mu\leq2,\\
x(\pm1)=x'(0^{\pm})&=0,
\end{split}\end{equation}
where ${}^{C}D_{0}^{\,\mu}x(t)={}^{C}D_{0^{+}}^{\,\mu}x(t)$ for $t\geq0$, ${}^{C}D_{0}^{\,\mu}x(t)={}^{C}D_{0^{-}}^{\,\mu}x(t)$ for $t\leq0$. Moreover, $f:(-1,\,1)\times(0,\infty)\rightarrow\mathbb{R}$ is continuous and singular at $t=-1$, $t=1$ and $x=0$. Here, ${}^{C}D_{0^{+}}^{\,\mu}$ and ${}^{C}D_{0^{-}}^{\,\mu}$, respectively, are Caputo fractional left and right derivatives of order $\mu$. We provide sufficient conditions for the existence of symmetric positive solutions of the Caputo fractional SBVP \eqref{cbvp}. By a symmetric positive solution $x$ of Caputo fractional SBVP \eqref{cbvp} we mean $x\in C[-1,\,1]$ satisfies \eqref{cbvp}, $x(t)=x(-t)$ for $t\in[-1,\,1]$ and $x(t)>0$ for $t\in(-1,\,1)$. For Caputo fractional SBVPs \eqref{cbvp}, two point boundary conditions (BCs) are $x(\pm1)=0$, whereas $x'(0^{\pm})=0$ are natural conditions followed by a function $x$ which is symmetric and concave on $[-1,\,1]$.

BVPs involving fractional order differentials have become an emerging area of recent research in science, engineering and mathematics, \cite{dhelm,kilbas,miller,podlubny}. Applying results of nonlinear functional analysis and fixed point theory, many articles have been devoted to the study of the existence of solutions for fractional differential equations under various type of BCs \cite{bailu,liluozhou,liangzhang,shizhang,suzhang,zhangwangsun}. Further, many authors considered the SBVPs involving fractional derivatives, for details see \cite{ars,lyons,mmz,stanek,szhx,yjx,zmw}. On the other side, many authors have established the existence of symmetric positive solutions for nonlinear BVPs of integer order derivatives of even order, for the references see \cite{averyhenderson,graefkong,jiangliuwu,kosmatov,ma}.

However, all the above works were obtained with fractional order left derivatives and symmetric solutions have been formulated for BVPs of even order derivatives only. Consequently, many aspects of the existence theory for the solutions of fractional order BVP needed to be explore, such as symmetric positive solutions. To the author's best knowledge, \eqref{cbvp} is a new form of fractional differential equation. Further, no paper in the existing literature has formulated the existence of symmetric positive solutions for BVPs containing fractional derivatives. An important features of this paper is that the Green's function of the Caputo fractional linear BVP corresponding to \eqref{cbvp} is symmetric.

The paper is organized as follows. In Section \ref{pre}, we recall some definitions from fractional calculus and some preliminary lemmas for construction of Green's function. Some properties properties of Green's function are also presented in the same section. In Section \ref{main}, by the use of Schauder's fixed point theorem and results of functional analysis, the existence of symmetric positive solutions is obtained in Theorem \ref{mainth}. Finally, an example is presented to illustrate our theory.

\section{preliminaries}\label{pre}

In this section, we shall state some necessary definitions and preliminary results. The following definitions and lemmas are known, \cite{kilbas,miller,podlubny}.

\begin{defn}
The Riemann-Liouville fractional left integral of order $\mu>0$ of a function $x:(0,\infty)\rightarrow\mathbb{R}$ is defined as
\begin{align*}
{}^{RL}I_{0^{+}}^{\,\mu}\,x(t)=\frac{1}{\Gamma(\mu)}\int_{0}^{t}(t-\tau)^{\mu-1}x(\tau)d\tau,\hspace{0.4cm}t\geq0.
\end{align*}
\end{defn}

\begin{defn}
The Riemann-Liouville fractional right integral of order $\mu>0$ of a function $x:(-\infty,0)\rightarrow\mathbb{R}$ is defined as
\begin{align*}
{}^{RL}I_{0^{-}}^{\,\mu}\,x(t)=\frac{1}{\Gamma(\mu)}\int_{t}^{0}(\tau-t)^{\mu-1}x(\tau)d\tau,\hspace{0.4cm}t\leq0.
\end{align*}
\end{defn}

\begin{defn}
The Caputo fractional left derivative of a function $x\in AC^{\,n}[0,\infty)$, $n\in\mathbb{N}$, of order $\mu\in(n-1,n]$ is defined as
\begin{align*}{}^{C}D_{0^{+}}^{\,\mu}x(t)=\frac{1}{\Gamma(n-\mu)}\int_{0}^{t}\frac{x^{(n)}(\tau)}{(t-\tau)^{\mu-n+1}}d\tau.\end{align*}
\end{defn}

\begin{defn}
The Caputo fractional right derivative of a function $x\in AC^{\,n}(-\infty,0]$, $n\in\mathbb{N}$, of order $\mu\in(n-1,n]$ is defined as
\begin{align*}{}^{C}D_{{0}^{-}}^{\,\mu}\,x(t)=\frac{(-1)^{n}}{\Gamma(n-\mu)}\int_{t}^{0}\frac{x^{(n)}(\tau)}{(\tau-t)^{\mu-n+1}}d\tau.\end{align*}
\end{defn}

\begin{lem}\label{rightsol}
For $\mu\in(n-1,n]$, $n\in\mathbb{N}$, the fractional differential equation ${}^{C}D_{0^{-}}^{\,\mu}\,x(t)+y(t)=0$, $t\leq0$, has a solution $x(t)=a_{1}+a_{2}\,t+a_{3}\,t^{2}+\cdots+a_{n}\,t^{n-1}\,-\,{}^{RL}I_{0^{-}}^{\,\mu}\,y(t)$, where $a_{i}\in\mathbb{R}$, $i=1,\cdots,n$.
\end{lem}

\begin{lem}\label{leftsol}
For $\mu\in(n-1,n]$, $n\in\mathbb{N}$, the fractional differential equation ${}^{C}D_{0^{+}}^{\,\mu}\,x(t)+y(t)=0$, $t\geq0$, has a solution $x(t)=b_{1}+b_{2}\,t+b_{3}\,t^{2}+\cdots+b_{n}\,t^{n-1}\,-\,{}^{RL}I_{0^{+}}^{\,\mu}\,y(t)$, where $b_{i}\in\mathbb{R}$, $i=1,\cdots,n$.
\end{lem}

The following lemmas are essential for Theorem \ref{mainth}. In Lemma \ref{lemir}, we construct the Green's function and formulate integral representation for a Caputo fractional linear BVP.

\begin{lem}\label{lemir}
Let $y\in C((-1,\,1),(0,\infty))$ satisfies $\int_{-1}^{\,1}(1-|t|)^{\mu-1}\,y(t)dt<\infty$, then the Caputo fractional linear BVP
\begin{equation}\label{lbvp}\begin{split}
{}^{C}D_{0}^{\,\mu}x(t)+y(t)&=0,\hspace{0.4cm}t\in(-1,\,1),\hspace{0.4cm}1<\mu\leq2,\\
x(\pm1)=x'(0^{\pm})&=0,
\end{split}\end{equation}
has integral representation
\begin{equation}\label{ir}
x(t)=\int_{-1}^{\,1}G(t,\tau)\,y(\tau)d\tau,\hspace{0.4cm}t\in[-1,\,1],
\end{equation}
where
\begin{equation}\label{gf}
G(t,\tau)=\frac{1}{\Gamma(\mu)}\begin{cases}
(1+\tau)^{\mu-1},\hspace{0.4cm}&-1\leq\tau\leq t\leq0,\\
(1+\tau)^{\mu-1}-(\tau-t)^{\mu-1},\hspace{0.4cm}&-1\leq t\leq\tau\leq0,\\
(1-\tau)^{\mu-1}-(t-\tau)^{\mu-1},\hspace{0.4cm}&\,\,\,\,\,0\leq\tau\leq t\leq 1,\\
(1-\tau)^{\mu-1},\hspace{0.4cm}&\,\,\,\,\,0\leq t\leq\tau\leq 1.
\end{cases}
\end{equation}
\end{lem}

\begin{proof}
Consider the Caputo fractional linear differential equation
\begin{equation}\label{exd}\begin{split}
{}^{C}D_{0}^{\,\mu}\,x(t)+y(t)=0,\hspace{0.4cm}t\in(-1,\,1),\hspace{0.4cm}1<\mu\leq2,
\end{split}\end{equation}
which implies that
\begin{align*}
{}^{C}D_{0^{-}}^{\,\mu}\,x(t)+y(t)&=0,\hspace{0.4cm}t\in[-1,0],\\
{}^{C}D_{0^{+}}^{\,\mu}\,x(t)+y(t)&=0,\hspace{0.4cm}t\in[0,\,1],
\end{align*}
which in view of Lemma \ref{rightsol} and Lemma \ref{leftsol}, leads to
\begin{align*}
x(t)&=a_{1}+a_{2}\,t-\frac{1}{\Gamma(\mu)}\int_{t}^{0}(\tau-t)^{\mu-1}y(\tau)d\tau,\hspace{0.4cm}t\in[-1,0],\\
x(t)&=b_{1}+b_{2}\,t-\frac{1}{\Gamma(\mu)}\int_{0}^{t}(t-\tau)^{\mu-1}y(\tau)d\tau,\hspace{0.4cm}t\in[0,\,1],
\end{align*}
which on employing the BCs \eqref{lbvp}, reduces to
\begin{align*}
x(t)&=\frac{1}{\Gamma(\mu)}\int_{-1}^{0}(\tau-1)^{\mu-1}y(\tau)d\tau-\frac{1}{\Gamma(\mu)}\int_{t}^{0}(\tau-t)^{\mu-1}y(\tau)d\tau,\hspace{0.4cm}t\in[-1,0],\\
x(t)&=\frac{1}{\Gamma(\mu)}\int_{0}^{1}(1-\tau)^{\mu-1}y(\tau)d\tau-\frac{1}{\Gamma(\mu)}\int_{0}^{t}(t-\tau)^{\mu-1}y(\tau)d\tau,\hspace{0.4cm}t\in[0,\,1],
\end{align*}
which can be expressed as
\begin{align*}
x(t)=\frac{1}{\Gamma(\mu)}\begin{cases}
\int_{-1}^{0}(1+\tau)^{\mu-1}y(\tau)d\tau-\int_{t}^{0}(\tau-t)^{\mu-1}y(\tau)d\tau,\hspace{0.4cm}&t\in[-1,0],\\
\int_{0}^{1}(1-\tau)^{\mu-1}y(\tau)d\tau-\int_{0}^{t}(t-\tau)^{\mu-1}y(\tau)d\tau,\hspace{0.4cm}&t\in[0,\,1],
\end{cases}
\end{align*}
which is equivalent to \eqref{ir}.
\end{proof}

In the following Lemma \ref{gbound}, we present some properties of the Green's function \eqref{gf}.

\begin{lem}\label{gbound}
The Green's function \eqref{gf} satisfies
\begin{itemize}
\item[(1).] $G:[-1,\,1]\times[-1,\,1]\rightarrow[0,\infty)$ is continuous and positive on $(-1,\,1)\times(-1,\,1)$.
\item[(2).] $G(t,\tau)=G(-t,-\tau)$ for all $(t,\tau)\in[-1,\,1]\times[-1,\,1]$.
\item[(3).] $G(t,\tau)\leq\frac{1}{\Gamma(\mu)}(1-|\tau|)^{\mu-1}$ for all $(t,\tau)\in[-1,\,1]\times[-1,\,1]$.
\item[(4).] $\int_{-1}^{\,1}G(t,\tau)d\tau=\frac{2(1-|t|^{\mu})}{\Gamma(\mu+1)}$, for all $t\in[-1,\,1]$.
\end{itemize}
\end{lem}

\begin{proof}
\begin{itemize}
\item[(1).] Clearly, the Green's function \eqref{gf} is continuous on $[-1,\,1]\times[-1,\,1]$ and positive on $(-1,\,1)\times(-1,\,1)$.
\item[(2).] The proof of $(2)$ is obvious.
\item[(3).] From the Green's function \eqref{gf}, we have
\begin{align*}
G(t,\tau)&=\frac{1}{\Gamma(\mu)}\begin{cases}
(1+\tau)^{\mu-1},\hspace{0.4cm}&-1\leq\tau\leq t\leq0,\\
(1+\tau)^{\mu-1}-(\tau-t)^{\mu-1},\hspace{0.4cm}&-1\leq t\leq\tau\leq0,\\
(1-\tau)^{\mu-1}-(t-\tau)^{\mu-1},\hspace{0.4cm}&\,\,\,\,\,0\leq\tau\leq t\leq 1,\\
(1-\tau)^{\mu-1},\hspace{0.4cm}&\,\,\,\,\,0\leq t\leq\tau\leq 1,
\end{cases}\\
&\leq\frac{1}{\Gamma(\mu)}\begin{cases}
(1+\tau)^{\mu-1},\hspace{0.4cm}&\hspace{0.45cm}-1\leq \tau\leq0,\hspace{0.4cm}-1\leq t\leq0,\\
(1-\tau)^{\mu-1},\hspace{0.4cm}&\hspace{0.75cm}0\leq \tau\leq 1,\hspace{0.7cm}0\leq t\leq 1,
\end{cases}\\
&=\frac{1}{\Gamma(\mu)}(1-|\tau|)^{\mu-1},\hspace{1.65cm}(t,\tau)\in[-1,\,1]\times[-1,\,1].
\end{align*}
\item[(4).] In view of \eqref{gf}, for $t\in[0,\,1]$, we have
\begin{align*}
\int_{0}^{1}G(t,\tau)d\tau=\frac{1}{\Gamma(\mu)}\int_{0}^{1}(1-\tau)^{\mu-1}d\tau-\frac{1}{\Gamma(\mu)}\int_{0}^{t}(t-\tau)^{\mu-1}d\tau=\frac{1-t^{\mu}}{\Gamma(\mu+1)}.
\end{align*}
Similarly in view of \eqref{gf}, for $t\in[-1,\,0]$, we have
\begin{align*}
\int_{-1}^{0}G(t,\tau)d\tau=\frac{1}{\Gamma(\mu)}\int_{-1}^{0}(1+\tau)^{\mu-1}d\tau-\frac{1}{\Gamma(\mu)}\int_{t}^{0}(\tau-t)^{\mu-1}d\tau=\frac{1-(-t)^{\mu}}{\Gamma(\mu+1)}.
\end{align*}
Thus for $t\in[-1,\,1]$, we have
\begin{align*}
\int_{-1}^{\,1}G(t,\tau)d\tau=\frac{2(1-|t|^{\mu})}{\Gamma(\mu+1)}.
\end{align*}
\end{itemize}
\end{proof}

\section{main result}\label{main}

Assume that
\begin{itemize}
\item[(A1).] For $x>0$, $f(t,x)=f(-t,x)$ for all $t\in(-1,\,1)$. There exist $q\in C(-1,\,1)$, $u\in C(0,\infty)$ decreasing, and $v\in C[0,\infty)$ increasing such that
\begin{align*}|f(t,x)|\leq q(t)(u(x)+v(x)),\hspace{0.4cm}t\in(-1,\,1),\hspace{0.4cm}x\in(0,\infty),\end{align*}
\begin{align*}\int_{-1}^{\,1}(1-|t|)^{\mu-1}\,q(t)dt<\infty,\text{ and }\int_{-1}^{\,1}(1-|t|)^{\mu-1}\,q(t)u\left(c\,(1-|t|^{\mu})\right)dt<\infty\,\text{ for }c>0.
\end{align*}
\item[(A2).] There exist a constant $R>\frac{2\,\gamma_{_{R}}}{\Gamma(\mu+1)}$ such that, for $t\in(-1,\,1)$ and $x\in(0,R]$, $f(t,x)\geq\gamma_{_{R}}$, where the parameter $\gamma_{r}$ is positive and decreasing for $r>0$. Moreover,
\begin{align*}\frac{R}{\chi_{_{R}}\left(1+\frac{q(R)}{p(R)}\right)}>1,\end{align*}
where
\begin{align*}
\chi_{r}=\int_{-1}^{\,1}(1-|t|)^{\mu-1}\,q(t)\,u\left(\frac{2\,\gamma_{_{r}}(1-|t|^{\mu})}{\Gamma(\mu+1)}\right)dt.
\end{align*}
\end{itemize}
In view of $(A2)$, choose $\varepsilon\in(0,R-\frac{2\,\gamma_{_{R}}}{\Gamma(\mu+1)}]$ such that
\begin{equation}\label{eps}
\frac{R-\varepsilon}{\chi_{_{R+\varepsilon}}\left(1+\frac{q(R+\varepsilon)}{p(R+\varepsilon)}\right)}\geq1.
\end{equation}
For $m\in\mathbb{N}$ with $\frac{1}{m}<\varepsilon$, consider the modified BVP
\begin{equation}\label{sp}\begin{split}
{}^{C}D_{0}^{\,\mu}x(t)+f\left(t,\min\{\max\{x(t)+\frac{1}{m},\frac{1}{m}\},R\}\right)&=0,\hspace{0.1cm}t\in(-1,\,1),\hspace{0.1cm}1<\mu\leq2,\\
x(\pm1)=0,\hspace{0.3cm}x'(0^{\pm})&=0,
\end{split}\end{equation}
which in view of Lemma \ref{lemir}, has integral representation
\begin{align*}
x(t)=\int_{-1}^{\,1}G(t,\tau)f\left(\tau,\min\{\max\{x(\tau)+\frac{1}{m},\frac{1}{m}\},R\}\right)d\tau,\hspace{0.4cm}t\in[-1,\,1].
\end{align*}
Let $X=\{x:x\in C[-1,\,1],\,x(t)=x(-t)\text{ for }t\in[-1,\,1]\}$. Define $T_{m}:X\rightarrow X$ by
\begin{equation}\label{mapt}
T_{m}x(t)=\int_{-1}^{\,1}G(t,\tau)f\left(\tau,\min\{\max\{x(\tau)+\frac{1}{m},\frac{1}{m}\},R\}\right)d\tau,\hspace{0.4cm}t\in[-1,\,1].
\end{equation}
Clearly, fixed points of $T_{m}$ are solutions of the Caputo fractional BVP \eqref{sp}.

\begin{thm}\label{mainth}
Assume that $(A1)$ and $(A2)$ hold. Then the Caputo fractional SBVP \eqref{cbvp} has a symmetric positive solution.
\end{thm}

\begin{proof}
In view of $(A1)$ and Schauder's fixed point theorem the map $T_{m}$ defined by \eqref{mapt} has a fixed point $x_{m}\in X$. Thus
\begin{equation}\label{fp}
x_{m}(t)=\int_{-1}^{\,1}G(t,\tau)f\left(\tau,\min\{\max\{x(\tau)+\frac{1}{m},\frac{1}{m}\},R\}\right)d\tau,\hspace{0.4cm}t\in[-1,\,1],
\end{equation}
which in view of $(A2)$ and Lemma \ref{gbound}, leads to
\begin{equation}\label{lb}
x_{m}(t)\geq\int_{-1}^{\,1}G(t,\tau)\gamma_{_{R}}d\tau\geq\frac{2\,\gamma_{_{R+\varepsilon}}\,(1-|t|^{\mu})}{\Gamma(\mu+1)}.
\end{equation}
Also \eqref{fp} in view of Lemma \ref{gbound}, $(A1)$, \eqref{lb} and \eqref{eps}, leads to
\begin{equation}\label{ub}\begin{split}
x_{m}(t)&\leq\frac{1}{\Gamma(\mu)} \int_{-1}^{\,1}(1-|\tau|)^{\mu-1}\,q(\tau)u\left(\min\{\max\{x_{m}(\tau)+\frac{1}{m},\frac{1}{m}\},R\}\right)\\
&\left(1+\frac{v(\min\{\max\{x_{m}(\tau)+\frac{1}{m},\frac{1}{m}\},R\})}{u(\min\{\max\{x_{m}(\tau)+\frac{1}{m},\frac{1}{m}\},R\})}\right)d\tau\\
&\leq \frac{1}{\Gamma(\mu)}\int_{-1}^{\,1}(1-|\tau|)^{\mu-1}\,q(\tau)\,u\left(\frac{2\,\gamma_{_{R+\varepsilon}}\,(1-|\tau|^{\mu})}{\Gamma(\mu+1)}\right)\left(1+\frac{v(R+\varepsilon)}{u(R+\varepsilon)}\right)d\tau\\
&=\chi_{_{R+\varepsilon}}\left(1+\frac{v(R+\varepsilon)}{u(R+\varepsilon)}\right)\leq R-\varepsilon
\end{split}\end{equation}
Consequently, from \eqref{lb} and \eqref{ub}, solution $x_{m}$ of BVP \eqref{sp} satisfies
\begin{equation}\label{xnt}x_{m}(t)=\int_{-1}^{\,1}G(t,\tau)f\left(\tau,x_{m}(\tau)+\frac{1}{m}\right)d\tau,\hspace{0.4cm}t\in[-1,\,1],\end{equation}
and
\begin{align*}
\frac{2\,\gamma_{_{R+\varepsilon}}\,(1-|t|^{\mu})}{\Gamma(\mu+1)}\leq x_{m}(t)<R,\hspace{0.4cm}t\in[-1,\,1],
\end{align*}
which shows that the sequence $\{x_{n}\}_{n=m}^{\infty}$ is uniformly bounded on $[-1,1]$. Moreover, since $G(t,\tau)$ is uniformly continuous on $[-1,\,1]\times[-1,\,1]$, by Lebesgue dominated convergence theorem, the sequence $\{x_{n}\}_{n=m}^{\infty}$ equicontinuous on $[-1,\,1]$. Thus by Arzela Ascoli Theorem the sequence $\{x_{n}\}_{n=m}^{\infty}$ is relatively compact and consequently there exist a subsequence $\{x_{n_{k}}\}_{k=1}^{\infty}$ converging uniformly to $x\in X$. Moreover, in view of \eqref{xnt}, we have
\begin{align*}
x_{n_{k}}(t)=\int_{-1}^{\,1}G(t,\tau)f\left(\tau,x_{n_{k}}(\tau)+\frac{1}{n_{k}}\right)d\tau,
\end{align*}
as $k\rightarrow\infty$, we obtain
\begin{equation}\label{intsol}
x(t)=\int_{-1}^{\,1}G(t,\tau)f(\tau,x(\tau))d\tau,
\end{equation}
which in view of Lemma \ref{lemir}, leads to
\begin{align*}
&{}^{C}D_{0}^{\,\mu}\,x(t)+f(t,x(t))=0,\hspace{0.4cm}t\in(-1,\,1),\\
&x(\pm1)=0,\hspace{0.4cm}x'(0^{\pm})=0.
\end{align*}
Also, ${}^{C}D_{0}^{\,\mu}\,x\in C(-1,\,1)$. Further, from \eqref{intsol} in view of $(A2)$ and Lemma \ref{gbound}, we have
\begin{align*}
x(t)=\int_{-1}^{\,1}G(t,\tau)f(\tau,x(\tau))d\tau\geq\int_{-1}^{\,1}G(t,\tau)\,\gamma_{_{R}}d\tau=\frac{2\,\gamma_{_{R}}\,(1-|t|^{\mu})}{\gamma(\mu+1)},
\end{align*}
which shows that $x(t)>0$ for $t\in(-1,\,1)$. Hence $x\in C[-1,\,1]$ with ${}^{C}D_{0}^{\,\mu}\,x\in C(-1,\,1)$ is a symmetric positive solution of the Caputo fractional SBVP \eqref{cbvp}.
\end{proof}

\begin{ex}

\begin{equation}\label{ebvp}\begin{split}
{}^{C}D_{0}^{\,^{1.9}}\,x(t)+\frac{\lambda}{(1-|t|^{0.9})^{0.9}}\left(\frac{1}{(x(t))^{0.9}}-x(t)+R\right)&=0,\hspace{0.4cm}t\in(-1,\,1),\\
x(\pm1)=0,\,\,x'(0^{\pm})&=0,
\end{split}\end{equation}
where
\begin{align*}
0<\lambda<\min\left\{\frac{R^{1.9}}{5.55871\times10^{10}\times\left(1+2R^{1.9}\right)^{10}},0.913678\times R^{1.9}\right\}.
\end{align*}
Here
\begin{align*}
f(t,x)=\frac{\lambda}{(1-|t|^{0.9})^{0.9}}\left(\frac{1}{x^{0.9}}-x+R\right).
\end{align*}
Choose
\begin{align*}
q(t)=\frac{\lambda}{(1-|t|^{0.9})^{0.9}},\hspace{0.4cm}u(x)=\frac{1}{x^{0.9}},\hspace{0.4cm}v(x)=x+R,\hspace{0.4cm}\gamma_{r}=\frac{\lambda}{r^{0.9}}.
\end{align*}
Then,
\begin{align*}
\int_{-1}^{\,1}(1-|t|)^{0.9}\,q(t)dt=2.12926\times\lambda,\hspace{0.4cm}\int_{-1}^{\,1}(1-|t|)^{0.9}\,q(t)\,u(c\,(1-|t|^{1.9}))dt=12.8761\times\frac{\lambda}{c^{0.9}}.
\end{align*}
Moreover,
\begin{align*}
|f(t,x)|\leq q(t)(u(x)+v(x)),\text{ for }t\in(-1,\,1),\hspace{0.4cm}x\in(0,\infty),
\end{align*}
\begin{align*}
f(t,x)\geq \gamma_{_{R}}\text{ for }t\in(-1,\,1),\hspace{0.4cm}x\in(0,R].
\end{align*}
Further,
\begin{align*}
\frac{R}{\chi_{_{R}}\left(1+\frac{v(R)}{u(R)}\right)}=\frac{R^{0.19}}{11.8713\times\left(1+2R^{1.9}\right)\lambda^{0.1}}>1,
\end{align*}
where
\begin{align*}
\chi_{r}=\int_{-1}^{\,1}(1-|t|)^{0.9}\,q(t)\,u\left(\frac{2\,\gamma_{_{r}}(1-|t|^{1.9})}{\Gamma(2.9)}\right)dt=11.8713\times\lambda^{0.1}\times r^{0.81}.
\end{align*}
Clearly, the assumptions $(A1)$ and $(A2)$ of Theorem \ref{mainth} are satisfied, therefore, the Caputo fractional SBVP \eqref{ebvp} has a symmetric positive solution.

\end{ex}


\begin{thebibliography}{99}

\bibitem{ars} R.P. Agarwal, D. O'Regan, S. Stanek, Positive solutions for Dirichlet problems of singular nonlinear fractional differential equations, J. Math. Anal. Appl. 371 (2010) 57-68.
    
\bibitem{averyhenderson} R.I. Avery, A.C. Henderson, Three symmetric positive solutions for a second-order boundary value problem, Appl. Math. Lett. 13 (2000) 1-7.

\bibitem{bailu} Z.B. Bai, H.S. L\"{u}, Positive solutions for boundary value problem of nonlinear fractional differential equation, J. Math. Anal. Appl. 311 (2005) 495-505.

\bibitem{dhelm} K. Diethelm, The Analysis of Fractional Differential Equations: An Application-Oriented Exposition Using Differential Operators of Caputo Type, Springer, 2010.
    
\bibitem{graefkong} J.R. Graef, L. Kong, Necessary and sufficient conditions for the existence of symmetric positive solutions of multi-point boundary value problems, Nonlinear Anal. 68 (2008) 1529-1552.
    
\bibitem{jiangliuwu} J. Jiang, L. Liu, Y. Wu, Symmetric positive solutions to singular system with multi-point coupled boundary conditions, Applied Mathematics and Computation 220 (2013) 536-548.

\bibitem{kilbas} A.A. Kilbas, H.M. Srivastava, J.J. Trujillo, Theory and Applications of Fractional Differential Equations, North-Holland Mathematics Studies, vol. 204, Elsevier Science B.V., Amsterdam, 2006.
    
\bibitem{kosmatov} N. Kosmatov, Symmetric solutions of a multi-point boundary value problem, J. Math. Anal. Appl. 309 (2005) 25-36.
    
\bibitem{liluozhou} C.F. Li, X.N. Luo, Yong Zhou, Existence of positive solutions of the boundary value problem for nonlinear fractional differential equations, Comp. Math. Appl. 59 (2010) 1363-1375.
    
\bibitem{liangzhang} S. Liang, J. Zhang, Positive solutions for boundary value problems of nonlinear fractional differential equation, Nonlinear Anal. 71 (2009) 5545-5550.
    
\bibitem{lyons} J.W. Lyons, J.T. Neugebauer, Positive solutions of a singular fractional boundary value problem with a fractional boundary condition, Opuscula Math. 37, no. 3 (2017) 421-434.
    
\bibitem{ma} H. Ma, Symmetric positive solutions for nonlocal boundary value problems of fourth order, Nonlinear Anal. 68 (2008) 645-651.
    
\bibitem{mmz} H. Maagli, N. Mhadhebi, N. Zeddini, Existence and estimates of positive solutions for some singular fractional boundary value problems, Abstr. Appl. Anal. (2014), Art. ID 120781.

\bibitem{miller} K.S. Miller, B. Ross, An Introduction to Fractional Calculus and Fractional Differential Equations, John Wiley and Sons, New York, 1993.

\bibitem{podlubny} I. Podlubny, Fractional Differential Equations. An Introduction to Fractional Derivatives, Fractional Differential Equations, to Methods of Their Solutions and Some of Their Applications, Mathematics in Science and Enginnering, vol. 198, Academic Press, San Diego, 1999.
    
\bibitem{shizhang} A. Shi and S. Zhang, Upper and lower solutions method and a fractional differential equation boundary value problem, Electron. J. Qual. Theory Differ. Equ. 30 (2009) 1-13.
    
\bibitem{stanek} S. Stanek, The existence of positive solutions of singular fractional boundary value problems, Comput. Math. Appl. 62 (2011) 1379-1388.

\bibitem{suzhang} X. Su, S. Zhang, Solutions to boundary value problems for nonlinear differential equations of fractional order, Elec. J. Diff. Equat. 26 (2009) 1-15.
    
\bibitem{szhx} S. Sun, Y. Zhao, Z. Han, M. Xu, Uniqueness of positive solutions for boundary value problems of singular fractional differential equations, Inverse Probl. Sci. Eng. 20 (2012) 299-309.
    
\bibitem{yjx} C. Yuan, D. Jiang, X. Xu, Singular positone and semipositone boundary value problems of nonlinear fractional differential equations, Math. Probl. Eng. (2009), Art. ID 535209.
    
\bibitem{zmw} X. Zhang, C. Mao, Y. Wu, H. Su, Positive solutions of a singular nonlocal fractional order differential system via Schauder's fixed point theorem, Abstr. Appl. Anal. (2014), Art. ID 457965.

\bibitem{zhangwangsun} X. Zhang, L. Wang, Q. Sun, Existence of positive solutions for a class of nonlinear fractional differential equations with integral boundary conditions and a parameter, Appl. Math. Comput. 226 (2014) 708-718.

\end{thebibliography}
\end{document}